\documentclass[11pt]{amsart}

\usepackage{amssymb}
\usepackage{graphicx, curves} 
\usepackage{eepic}
\usepackage{amsmath}
\usepackage{color}
\usepackage{fullpage}

\newtheorem{theorem}{Theorem}[section]
\newtheorem{lemma}[theorem]{Lemma}

\theoremstyle{definition}
\newtheorem{definition}[theorem]{Definition}
\newtheorem{corollary}[theorem]{Corollary}
\newtheorem{proposition}[theorem]{Proposition}
\newtheorem{example}[theorem]{Example}

\numberwithin{equation}{section}

\theoremstyle{remark}
\newtheorem{remark}[theorem]{Remark}

\numberwithin{equation}{section}

\newcommand{\C}{\mathcal{C}}
\newcommand{\F}{\mathcal{F}}

\newcommand{\nP}{\mathcal{P}}
\newcommand{\nQ}{\mathcal{Q}}
\newcommand{\nD}{\mathcal{D}}

\newcommand{\im}{\mathrm{im}}

\newcommand{\field}{{\bf k}}

\DeclareMathOperator{\rlk}{rlk}
\DeclareMathOperator{\ohom}{OHOM}
\DeclareMathOperator{\lk}{lk}

\title[]{Cellular Resolutions of Ideals Defined by Simplicial Homomorphisms}

\author{Benjamin Braun}
\address{Department of Mathematics\\
         University of Kentucky\\
         Lexington, KY 40506--0027}
\email{benjamin.braun@uky.edu}
\urladdr{http://www.ms.uky.edu/~braun/}

\author{Jonathan Browder}
\address{Department of Mathematics \\
	University of Washington  \\
	Box 354350 \\
	Seattle, WA 98195}
\email{browder@math.washington.edu}
\urladdr{http://www.math.washington.edu/~browder/}

\author{Steven Klee}
\address{Mathematical Sciences Building \\
	One Shields Ave. \\
	University of California \\
	Davis, CA 95616}
\email{klee@math.ucdavis.edu}
\urladdr{http://www.math.ucdavis.edu/~klee/}

\date{\today}

\thanks{Benjamin Braun was partially supported by NSF award DMS-0758321.  
Jonathan Browder was partially supported by NSF VIGRE award DMS-0354131. 
Steven Klee was partially supported by NSF VIGRE award DMS-0636297.
}

\begin{document}

\maketitle

\begin{abstract}
In this paper we introduce the class of ordered homomorphism ideals and prove that these ideals admit minimal cellular resolutions constructed as homomorphism complexes.
As a key ingredient of our work, we introduce the class of cointerval simplicial complexes and investigate their combinatorial and topological properties.
As a concrete illustration of these structural results, we introduce and study nonnesting monomial ideals, an interesting family of combinatorially defined ideals.
\end{abstract}

%\tableofcontents

%%%%%%%%%%%%%%%%%%%%%%%%%%%%%%%%%%%%%

\section{Introduction}

The use of techniques from topological combinatorics in combinatorial commutative algebra has been growing over the past decade.
Recall that a monomial ideal $I$ in the polynomial ring $S:=\field[x_1,\ldots,x_n]$ may be analyzed homologically through a free resolution of $I$, where a free resolution of $I$ is an exact chain complex of free $S$-modules terminating in $S/I$.
The resolution is called minimal if the rank of each of the free modules is minimized over all free resolutions of $I$; it is known that these ranks are independent of the choice of minimal free resolution.
If the maps for the resolution can be recovered from the cellular chain complex of a labeled cell complex as discussed in Section~\ref{background}, then the resolution is called cellular.

Batzies and Welker \cite{BatziesWelker} applied discrete Morse theory to construct minimal cellular resolutions from non-minimal cellular resolutions.
This has led to the development of algebraic discrete Morse theory by J\"{o}llenbeck and Welker \cite{JollenbeckWelker} (and developed independently by Sk\"{o}ldberg \cite{Skoldberg}).
Dochtermann and Engstr\"{o}m \cite{DochtermannEngstromEdgeIdeals} used standard techniques from topological combinatorics to give streamlined proofs of various results regarding free resolutions of edge ideals of graphs.

In recent work of Corso and Nagel \cite{CorsoNagel1,CorsoNagel2} and Nagel and Reiner \cite{NagelReiner}, specific families of ideals are studied that admit minimal cellular resolutions whose underlying cell complexes are prodsimplicial, where a \emph{prodsimplicial complex} is a polyhedral complex whose cells are products of simplices.
A major source of prodsimplicial complexes are the graph homomorphism complexes introduced by Lov\'{a}sz as a topological tool for providing lower bounds on graph chromatic numbers, as described in the recent monograph by Kozlov \cite{KozlovBook}; the homomorphism complex construction can be defined in a general context and is a source of subtle and fascinating topological structures in combinatorics.
Dochtermann and Engstr\"{o}m \cite{DochtermannEngstromCellular} introduced the use of the homomorphism complex construction to study (hyper-)edge ideals of a class of (hyper-)graphs they called cointerval.
The resulting prodsimplicial resolutions were reformulations and extensions of constructions arising in the work of Corso, Nagel, and Reiner, particularly of the ``complexes of boxes'' resolutions \cite[Section 3.3]{NagelReiner}.

With the present work, we offer three contributions to this fruitful interaction of topological combinatorics and combinatorial commutative algebra.
\begin{enumerate}
\item We introduce ordered homomorphism ideals, a family of ideals that admit minimal prodsimplicial resolutions constructed as homomorphism complexes.
Our hope is that the explicit homomorphism-based approach developed here will provide a simplifying language and useful technical tools for other researchers investigating ideals admitting prodsimplicial resolutions.
\item We investigate combinatorial and topological properties of cointerval simplicial complexes, one of the two inputs required by the homomorphism approach.
While the cointerval complexes introduced here are related to the cointerval hypergraphs introduced by Dochtermann and Engstr\"{o}m, neither class subsumes the other.
\item We illustrate these techniques by introducing nonnesting monomial ideals, a family of combinatorially defined ideals that may be studied with homomorphism complexes.
\end{enumerate}

Our paper is organized as follows; any terminology used here is defined in the relevant section.
In Section~\ref{background} we provide basic definitions and establish notational conventions.  
We describe the homomorphism complex construction and introduce the ordered homomorphism complex between two simplicial complexes whose vertex sets are linearly ordered. 
We also introduce the ordered homomorphism ideal associated to two ordered simplicial complexes and explain how it is connected to the homomorphism complex construction.

In Section~\ref{resolutions} we define cointerval simplicial complexes and prove that ordered homomorphism complexes support minimal linear cellular resolutions of ordered homomorphism ideals when the target complex is cointerval.
We offer two proofs of Theorem~\ref{main}: one is an inductive proof based on the number of homomorphisms considered, while the other proof uses poset closure operators and is motivated by standard techniques in the theory of homomorphism complexes.

In Section~\ref{properties} we investigate combinatorial and topological properties of cointerval simplicial complexes.  
In particular, we prove that cointerval complexes are vertex decomposable.
We also remark that shifted complexes are cointerval and that matroid complexes are not cointerval in general.

In Section~\ref{families} we illustrate these techniques by introducing nonnesting monomial ideals, a family of combinatorially defined ideals that may be approached with homomorphism complexes.
These ideals are related to the theory of nonnesting partitions, a Catalan-enumerated structure.
We show that Betti numbers of such ideals are related to poset order ideals in a recently-studied poset whose elements are Dyck paths of fixed length.

\subsection{Acknowledgements}

We thank Alberto Corso, Anton Dochtermann, Alex Engstr\"{o}m, Uwe Nagel, Isabella Novik, and Vic Reiner for thoughtful conversations and suggestions.

%%%%%%%%%%%%%%%%%%%%%%%%%%%%%%

\section{Background and Fundamental Constructions}\label{background}

In this section we review the basics of simplicial and polyhedral cell complexes and free resolutions of monomial ideals.
We also provide an introduction to homomorphism complexes and introduce ordered homomorphism ideals and cointerval complexes.
There are many excellent references for topological/enumerative combinatorics and combinatorial commutative algebra \cite{BjornerSurvey,KozlovBook,MillerSturmfels,StanleyVol1,StanleyVol2}.

%%%%%%%%%%%%%%%%%%%%%%%%%%%%%%%%%%%%%%%%%%%%

\subsection{Simplicial and polyhedral cell complexes}

\begin{definition}
A \emph{simplicial complex} $H$ on vertex set $V(H) = [n]$ is a collection of subsets $\sigma \subset [n]$, called \emph{faces}, that is closed under inclusion, i.e., if $\sigma \in H$ and $\tau \subset \sigma$, then $\tau \in H$ as well. 
\end{definition}

We make no distinction throughout this work between an abstract simplicial complex and its geometric realization as a topological space.
The \emph{dimension} of a face $\sigma \in H$ is $\dim \sigma:= |\sigma|-1$, and the \emph{dimension} of $H$ is $\dim H:=\max\{\dim \sigma: \sigma \in H\}$.  
If $H$ is a simplicial complex and $\sigma$ is a face of $H$, the \emph{link} of $\sigma$ in $H$ is the simplicial complex 
\[\lk_{H}(\sigma):=\{\tau \in H: \sigma \cap \tau = \emptyset, \sigma \cup \tau \in H\} \, .\]
When $\sigma$ consists of a single vertex, we generally write $\lk_{H}(v)$ instead of $\lk_{H}(\{v\})$.
We will make repeated use of the following standard fact.

\begin{proposition} \label{link-property}
Let $H$ be a simplicial complex.  If $\sigma \cup \sigma'$ is a face of $H$, then $$\lk_H(\sigma \cup \sigma') = \lk_{\lk_H(\sigma)}(\sigma').$$
\end{proposition}

\begin{definition}
A \emph{polyhedral cell complex} $X$ is a finite collection of convex polytopes in a real vector space $V$, called \emph{faces} of $X$, satisfying the following two properties:
\begin{itemize}
\item If $\tau$ is a polytope in $X$ and $\eta$ is a face of $\tau$, then $\eta$ is a polytope in $X$.
\item If $\tau$ and $\sigma$ are both polytopes in $X$, then $\tau\cap \sigma$ is a face of both $\tau$ and $\sigma$.
\end{itemize}
\end{definition}

We often view a polyhedral cell complex $X$ as a cell complex in the sense of CW-complexes in topology.
Consider the chain complex $\C$ over a field $\field$ for $X$ where the $i$-th vector space in the complex is the vector space $C_i$ with basis $\{e_\tau\}$ indexed by the $i$-dimensional cells of $X$, denoted $X^{(i)}$.
The map $\partial_i$ in $\C$ is obtained by extending linearly the map 
\[\partial_ie_\tau = \sum_{\sigma\in X^{(i-1)}}[\tau:\sigma]e_{\sigma}\]
where $[\tau:\sigma]$ is the incidence number of $\tau$ and $\sigma$ in $X$.  
See Hatcher \cite[Chapter 2]{Hatcher} for more on computing cellular chain complexes, with the caveat that Hatcher denotes the incidence number by $d_{\tau,\sigma}$.
The \emph{homology groups} of $X$ (actually vector spaces in our context) are defined to be $H_i(X,\field):=\ker(\partial_{i-1})/\im(\partial_i)$.  
We say $X$ is \emph{acyclic over $\field$} if the homology groups for $X$ are all zero when the vector spaces $C_i$ are $\field$-vector spaces.
The chain complex and homology groups of $X$ will be useful when constructing cellular resolutions of ideals from $X$.

%%%%%%%%%%%%%%%%%%%%%%%%%%%%%%%%%%%%%%%%%%

\subsection{Resolutions of monomial ideals}

Our presentation in this section follows the textbook of Miller and Sturmfels \cite{MillerSturmfels}.
A \emph{monomial ideal} in the ring $S:=\field[x_1,\ldots,x_n]$ is an ideal $I$ generated by monomials.
A \emph{(free) resolution} of a monomial ideal $I$ is an exact chain complex of free $S$-modules $\displaystyle F_i=\oplus_{j=1}^{\beta_i}S$ of the form
\[0\leftarrow S/I \leftarrow F_0 \leftarrow F_1 \leftarrow F_2 \leftarrow \cdots
\]
with maps $\phi_i:F_i\rightarrow F_{i-1}$.
We assume throughout that all our resolutions are graded, meaning the $\phi_i$ are degree-preserving; typically to make the resolution graded, the grading on the $S$-summands are shifted, but we will suppress this shift in our notation and write only $S$.
A resolution is \emph{minimal} if the rank $\beta_i$ of each $F_i$ is minimized over all possible $\beta_i$'s appearing in a resolution of $I$.
The $\beta_i$'s are called the \emph{Betti numbers} of $I$ and are a fundamental homological invariant with important applications in algebra, combinatorics, algebraic geometry, algebraic statistics, and other areas.
A resolution is \emph{cellular} if there is a cell complex $X$ whose cellular chain complex supports the resolution, in the sense that with an appropriate labeling of the faces of $X$ by monomials in $S$, our resolution can be recovered from the cellular chain complex of $X$ in the following way.

As we will only consider polyhedral cellular resolutions in the present work, we restrict our definitions to this case; for a quick introduction to the more general case of CW-supported resolutions, see \cite{BatziesWelker}.
In order for $X$ to support a resolution of $I$, each vertex of $X$ must correspond to a unique monomial generator $x^\alpha$ of $I$, which we consider as a label on the vertex.
To each face $\tau$ of $X$ we assign as a label the least common multiple of the labels on the vertices of $\tau$, which we denote by $x_\tau$.
We say that this labeling of $X$ supports a resolution of $I$ if $I$ is resolved by the chain complex of free $S$-modules whose $i$-th free module has summands corresponding to the elements of $X^{(i)}$ and whose $i$-th map is defined on generators of the summands by
\[\partial_ie_\tau = \sum_{\sigma\in X^{(i-1)}}[\tau:\sigma]\frac{x_{\tau}}{x_{\sigma}}e_{\sigma} \, ,\]
where here $e_{\sigma}$ indicates the generator of the summand corresponding to $\sigma$.
We denote this resolution of $I$, if it exists, by $\F(X)$.
Fortunately, there is a simple criterion that allows us to check when a complex $X$ with labeled vertices supports a resolution of the ideal generated by the vertex labels.

\begin{proposition}\label{acyclic}{\rm \cite[Proposition 4.5]{MillerSturmfels}} Given a polyhedral complex $X$ with vertices labeled by the generators of $I$, $\F(X)$ is a resolution of $S/I$ if and only if for every monomial $x^\alpha$ in $S$, the subcomplex $X_{\leq x^\alpha}$ is acyclic over $\field$, where $X_{\leq x^\alpha}$ denotes the subcomplex of $X$ consisting of faces with labels dividing $x^\alpha$.
\end{proposition}

It is straightforward to check that $\F(X)$ is minimal if whenever $\sigma\subset\tau$ is a strict inclusion of faces, then $x_{\tau}\neq x_{\sigma}$.

\begin{remark}
Minimal cellular resolutions of monomial ideals were introduced by Bayer and Sturmfels \cite{BayerSturmfels} and have since been the subject of extensive investigation.
It is interesting that not all resolutions of a monomial ideal are cellular in nature.
Velasco \cite{Velasco} provides examples of such resolutions and discusses differences between resolutions that are supported on simplicial complexes, polyhedral cell complexes, and arbitrary cell complexes, respectively.
\end{remark}

%%%%%%%%%%%%%%%%%%%%%%%%%%%%%%%%%%%%%%%%%%%%

\subsection{Homomorphism complexes and ordered homomorphism ideals}

The homomorphism complex construction was created to provide topological lower bounds on chromatic numbers of graphs.
The motivation for the general construction is the following.
Suppose that one has two finite sets $A$ and $B$, and a set of allowed maps between them called homomorphisms.
In many situations, the existence of a single homomorphism forces a relation to hold between certain invariants of our sets, for example: if $A$ and $B$ are graphs and there exists a graph homomorphism from $A$ to $B$, then $\chi(A)\leq \chi(B)$; if $A$ and $B$ are simplicial complexes and $\phi$ maps $A$ to $B$ non-degenerately, then $\dim(A)\leq \dim(B)$.

The homomorphism complex construction is an attempt to capture additional information about the relations on these invariants by considering the relationships between and among various families of homomorphisms in a topological setting.
Given a set $M$ of homomorphisms from $A$ to $B$, the cells of the homomorphism complex associated to $M$ correspond to all subsets of $M$ that are obtained by specifying a list of elements from $B$ for each element of $A$ and allowing homomorphisms to be built by independently selecting one element from each of these lists.
One can view the homomorphism complex as identifying (up to homotopy) families of homomorphisms that arise in this way.

\begin{definition}
Let $A$ and $B$ be finite sets.  
Let $M:=\{ \phi:A\rightarrow B \}$  be a collection of set maps called homomorphisms.  
Let $P(A,B)$ be the polyhedral cell complex $\prod_{i\in A}\Delta_B$, where $\Delta_B$ is the simplex with vertices labeled by elements of $B$.  
We record a face of $P(A,B)$ by an $|A|$-tuple of the form $X=(X_i)_{i\in A}$, with each $\emptyset \neq X_i\subseteq B$.
The faces are partially ordered by $X\leq Y$ if and only if $X_i\subseteq Y_i$ for all $i\in A$.  
The \emph{homomorphism complex} $\mathrm{HOM}_M(A,B)$ is the subcomplex of $P(A,B)$ consisting of all cells labeled by $(W_i)_{i\in A}\in P(A,B)$ such that if $\psi:A\rightarrow B$ satisfies $\psi(i)\in W_i$ for all $i\in A$ then $\psi\in M$.  
We call such a $(W_i)_{i\in A}\in P(A,B)$ a \textit{multi-homomorphism}, and we call $A$ the \textit{test object} and $B$ the \textit{target object} for the complex.
\end{definition}

The cells in $\mathrm{HOM}_M(A,B)$ are all products of simplices; such complexes are called \emph{prodsimplicial} and are special examples of polyhedral cell complexes.
We now introduce a new family of ideals defined by simplicial homomorphisms.

\begin{definition}\label{ordereddef}
Let $G$ and $H$ be simplicial complexes on finite totally ordered vertex sets $[n]$ and $[m]$, respectively.
Define an \emph{ordered simplicial homomorphism} from $G$ to $H$ to be a map $\phi:[n]\rightarrow [m]$ such that $\phi$ is weakly order preserving on $[n]$ and such that $\phi$ maps faces of $G$ to faces of $H$ non-degenerately, i.e. $\phi$ preserves face dimension.
We define the \emph{ordered homomorphism ideal} $I_{G,H,ord}$ to be the monomial ideal in $k[x_1,\ldots,x_m]$ generated by the monomials $\prod_{i\in[n]}x_{\phi(i)}$ where $\phi$ ranges over all ordered simplicial homomorphisms from $G$ to $H$.
\end{definition}

For a graph $G$ with any ordering of its vertices, the edge ideal of $G$ arises as $I_{K_2,G,ord}$.
The usefulness of the homomorphism complex construction for $I_{G,H,ord}$ is that the monomial generators correspond to homomorphisms between simplicial complexes, and the syzygies for these generators can in some cases be captured precisely by the topological interaction of these homomorphisms in a homomorphism complex.
To emphasize the role of the underlying ordering on the simplicial complexes for these ideals and their resolutions, we make the following definition.

\begin{definition}
Given $G$ and $H$ as in Definition~\ref{ordereddef}, denote by $\ohom(G,H)$ the homomorphism complex with test object $G$, target object $H$, and morphisms all ordered simplicial homomorphisms from $G$ to $H$.
\end{definition}

\begin{example}
Let $K$ be the simplicial complex on vertex set $\{a<b<c<d<e\}$ with facets $\{a,b,d\}$ and $\{c,d\}$.  
Let $L$ be the simplicial complex on vertex set $\{1<2<3<4<5\}$ with facets $\{1,2,4\},\{1,2,5\},\{3,4\},\{3,5\}$.
It follows that
\[I_{K,L,ord}=\left( x_1x_2^2x_4,x_1x_2x_3x_4,x_1x_2^2x_5,x_1x_2x_3x_5 \right).\]
The complex $\ohom(L,K)$ is a two-dimensional square.

Note that if we remove the two-dimensional face $\{1,2,5\}$ from $L$ to produce a complex $L'$, then the resulting ideal is
\[I_{K,L',ord}=\left( x_1x_2^2x_4,x_1x_2x_3x_4 \right) \, , \]
for which the complex $\ohom(K,L')$ is a line segment.
\end{example}

\begin{example}\label{bigexample}Let $G$ be the graph on vertex set $\{a<b<c\}$ with edge $\{a,c\}$.  Let $H$ be the graph on vertex set $\{1<2<3<4\}$ containing all edges except $\{3,4\}$. 
It follows that
\[I_{G,H,ord}=\left( x_1^2x_2,x_1^2x_3,x_1^2x_4,x_2^2x_3,x_1x_2^2,x_1x_2x_3,x_1x_2x_4,x_1x_3^2,x_1x_3x_4,x_1x_4^2,x_2x_3^2 \right) .\]
The complex $\ohom(G,H)$, shown in Figure~\ref{ex1hom}, has ten vertices and the following facets: two 3-dimensional prisms that share a common rectangle, one 2-dimensional square sharing an edge with one of the prisms, and a 3-dimensional simplex sharing a triangle with one of the prisms.
\end{example}

\begin{figure}[ht]
\begin{center}
\input{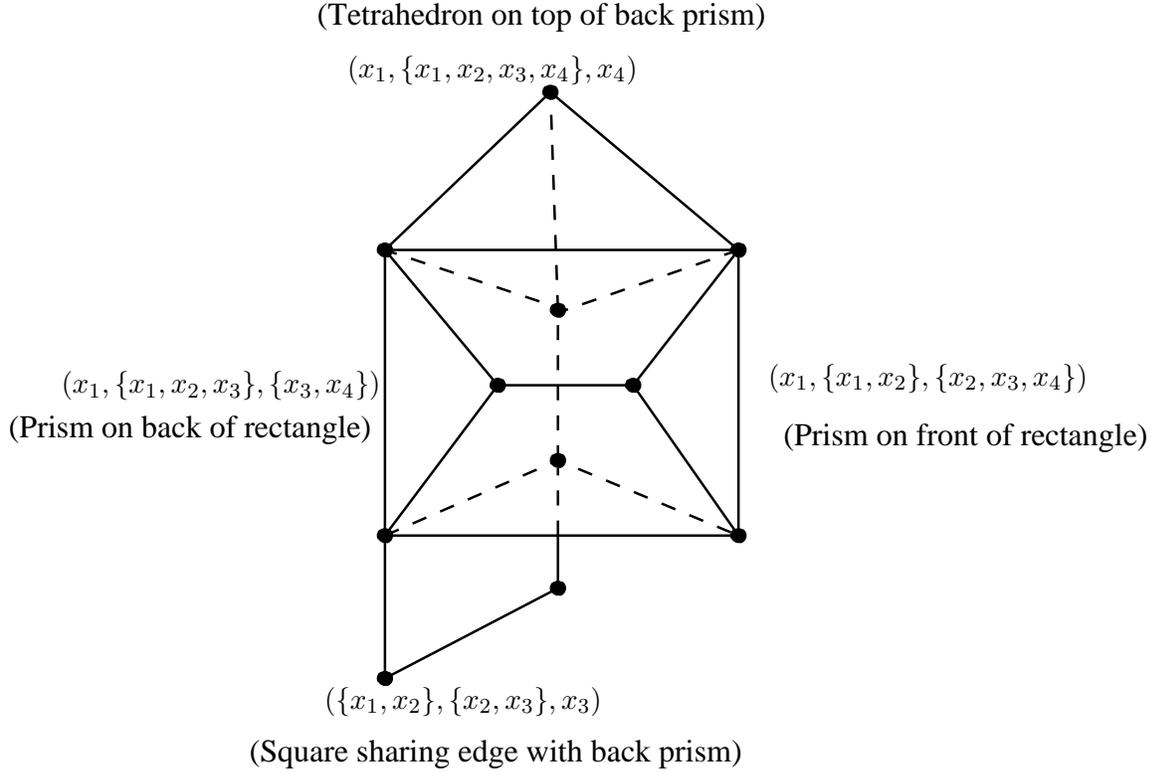}
\end{center}
\caption{$\ohom(G,H)$ from Example~\ref{bigexample}}
\label{ex1hom}
\end{figure}

It is straightforward that for a cell $\tau=(W_i)_{i\in [n]}\in \ohom(G,H)$, the least common multiple $x_\tau$ of the vertices of $\tau$ is given by the monomial 
\[ x_\tau = \prod_{\substack{i\in [n] \\ j\in W_i}}x_{j} \, .\]
Because of this, we sometimes refer to cells by both their names and their labels, e.g. $\tau$ or $x_\tau$.

We will need to order the monomial generators in our ideal by the revlex order, defined as follows.

\begin{definition}
For two monomials $x^\alpha$ and $x^\beta$ in the variables $x_1,\ldots,x_n$, we say that $x^\alpha$ is revlex larger than $x^\beta$ if the degrees of the monomials are the same and the variable with greatest index that appears in $x^\alpha / x^\beta$ has a negative exponent.
\end{definition}

We may consider the revlex order instead on ordered sequences of non-negative integers, where the order is obtained by considering the sequences as exponent vectors for monomials.
In a foreshadowing of our application of Proposition~\ref{acyclic}, we define the following complexes.

\begin{definition}
Let $\alpha:=(\alpha_1,\ldots,\alpha_m)$ be an integer vector and $x^\beta$ be a monomial of degree $n$. 
Denote by $\ohom(G,H)_\alpha$ the subcomplex of $\ohom(G,H)$ consisting of those faces whose monomial labels satisfy $\mathrm{deg}(x_i)\leq \alpha_i$, where by $\mathrm{deg}(x_i)$ we mean the degree of $x_i$ in the monomial.  
Denote by $\ohom(G,H)^{\geq \beta}$ the subcomplex of $\ohom(G,H)$ whose faces are those with vertices labeled by monomials $x^\alpha$ satisfying $x^\alpha\geq x^\beta$ in the revlex order.
\end{definition}

%%%%%%%%%%%%%%%%%%%%%%%%%%%%%%%%%%%%%%%

\section{Cointerval simplicial complexes and resolutions of ideals}\label{resolutions}

In this section we define cointerval simplicial complexes and prove that when a cointerval complex is used as a target object for the $\ohom$-construction, the resulting $\ohom$-complex supports a resolution of the corresponding ordered homomorphism ideal.

\begin{definition} For $H$ a simplicial complex on totally ordered vertex set $V(H)$ and $\tau \in H$ such that $k$ is the largest element of $\tau$, we define the \emph{right link} of $\tau$ in $H$ to be $\rlk_H(\tau) := (\lk_H(\tau))_{\{i > k\}}$, i.e. the induced subcomplex of $\lk(\tau)$ on the vertices ``to the right'' of $\tau$.
\end{definition}

\begin{definition} Let $H$ be a $d$-dimensional simplicial complex with totally ordered vertex set $V(H)$. 
We inductively define $H$ to be \emph{cointerval} if either $d =-1$ or $d \geq 0$ and
\begin{enumerate}
\item For each $i \in V(H)$, $\rlk_H(i)$ is cointerval.
\item For $i, j \in V(H)$ and $i < j$, $\rlk_H(j)$ is a subcomplex of $\rlk_H(i)$.
\end{enumerate}
\end{definition}

Dochtermann and Engstr\"{o}m \cite{DochtermannEngstromCellular} introduced the notion of cointerval hypergraphs to describe a family of hypergraphs where the homomorphism complex construction would support cellular resolutions for the special case where the test hypergraph is a hyperedge and the target is cointerval.
We refer to their cointerval hypergraphs as DE-cointerval hypergraphs to avoid confusion with our cointerval complexes.
The following example shows that not every DE-cointerval hypergraph is a cointerval complex when considered as a pure simplicial complex.
Similarly, as non-pure cointerval complexes are easy to construct, not every cointerval complex arises as a DE-cointerval hypergraph.

\begin{example}\label{equivalence}
Let $H$ be the hypergraph on $[6]$ with facets given by $\{2,5,6\}$ and $\{1,a,b\}$ for all $a,b \in \{3,4,5,6\}$.  
Then $H$ is DE-cointerval, but is not a cointerval complex as one can see by comparing the right links of $2$ and $3$.
To make this cointerval as a complex requires that we add the face $\{2,4,6\}$. 
\end{example}

We now state the main theorem of this section, which asserts that $\ohom$ complexes are contractible when the target complex is cointerval.
It follows from this that $\ohom(G,H)_{\alpha}$ is acyclic for any $\alpha$.

\begin{theorem}  \label{main} 
Let $G$ and $H$ be simplicial complexes on $[n]$ and $[m]$, respectively, and suppose $H$ is cointerval. 
Let $\alpha = (\alpha_1, \ldots, \alpha_m)$ be an integer vector and $x^\beta$ a degree $n$ monomial.  
Let $\Delta := \ohom(G,H)_{\alpha}^{\geq \beta}$. 
If $\Delta$ is nonempty, then $\Delta$ is contractible.
\end{theorem}

Proposition~\ref{acyclic} and Theorem~\ref{main} imply that the complex $\ohom(G,H)^{\geq \beta}$ supports a resolution of the ideal generated by a terminal segment up to $x^\beta$ of monomials under revlex ordering from the generating set of $I_{G,H,ord}$.

\begin{corollary}\label{maincor}
Given $G,H$ and $x^\beta$ as in Theorem~\ref{main}, let $I^{\geq \beta}_{G,H,ord}$ be the ideal whose generators are the monomial generators of $I_{G,H,ord}$ that are greater than or equal to $x^\beta$ in the revlex order.
The complex $\ohom(G,H)^{\geq \beta}$ supports a minimal cellular resolution of $I^{\geq \beta}_{G,H,ord}$.
\end{corollary}

\begin{remark}
All of our resolutions are linear, in the sense that the entries in the matrices for the maps $\phi_i$ in the cellular resolution of $I^{\geq \beta}_{G,H,ord}$ contain only degree one monomials.
The linearity of resolutions resulting from homomorphism complexes is an immediate consequence of their construction.
\end{remark}

In the remainder of this section, we provide two proofs of Theorem~\ref{main}.
In these proofs we refer to Lemmas~\ref{induced-cointerval},~\ref{rlk-cointerval}, and~\ref{links-cointerval}, technical lemmas whose verification we leave to Section~\ref{properties}.

%%%%

\subsection{Proof by induction}

Our first proof of Theorem~\ref{main} is inductive and uses techniques similar to those used by Nagel and Reiner \cite{NagelReiner} when analyzing complexes-of-boxes resolutions.
The key to our argument is Lemma~\ref{swap}, which allows us to move from one vertex of our $\ohom$ complex to another by ``swapping'' the image of a vertex with another compatible image.

\begin{lemma}[Swapping Lemma] \label{swap} 
Let $\Delta$ be as in Theorem \ref{main}, and suppose $\gamma$ and $\phi$ are homomorphisms corresponding to distinct vertices of $\Delta$. 
Let $k$ be the smallest element of $[n]$ such that $\gamma(k) \neq \phi(k)$.
Suppose  $\gamma(k) > \phi(k)$, and let $\gamma'$ be given by
\[
\gamma'(x) = \begin{cases} \gamma(x), & x \neq k\\ \phi(k) & x=k\end{cases}.
\]
Then $\gamma'$ is a vertex of $\Delta$, and $x_{\gamma'} \geq x_{\gamma}$.

\begin{proof}
The second claim is immediate, so we need only show that $\gamma'$ is an ordered simplicial homomorphism from $G$ to $H$ such that the degree of $x_i$ in $x_{\gamma'}$ is no greater than $\alpha_i$.

We first check that $\gamma'$ is order preserving. 
Suppose $i <j$, $i,j \in [n]$. 
If $i, j \neq k$, then $\gamma'(i) = \gamma(i)$ and $\gamma'(j) = \gamma(j)$, so as $\gamma$ is order preserving, $\gamma'(i) \leq \gamma'(j)$. 
On the other hand, suppose $i < k$. 
Then $\gamma(i) = \phi(i)$ (as $k$ is the smallest integer on which $\phi$ and $\gamma$ disagree), and thus as $\phi$ is order preserving, $\gamma'(i) = \gamma(i) = \phi(i) \leq \phi(k) = \gamma'(k)$. 
Finally, suppose $j >k$. 
Then $\gamma'(j) = \gamma(j) \geq \gamma(k) > \phi(k) = \gamma'(k)$. 

We next show that $\gamma'$ maps faces of $G$ to faces of $H$ non-degenerately. 
Let $F$ be an $i$-face of $G$.  
If $k \notin F$, then $\gamma'(F) = \gamma(F)$, an $i$-face of $H$. 
If $k \in F$, we may write $F = A \cup \{k\} \cup B$, where each element of $A$ is less than $k$ and each element of $B$ is greater than $k$.

Observe that $\gamma'(A) = \gamma(A)$ and $\gamma'(B) = \gamma(B)$ are disjoint sets in $[m]$ of size $|A|$ and $|B|$, respectively. 
Thus $\gamma'(F)$ will have size $|F|$ as long as $\gamma'(k) \notin \gamma(A) \cup \gamma(B)$. 
Since $\gamma$ is order preserving, for any $b \in B$, $\gamma'(k) < \gamma(k) \leq \gamma(b)$, so $\gamma'(k) \notin \gamma(B)$. 
On the other hand, for $a \in A$, $\gamma(a) = \phi(a) \neq \phi(k) = \gamma'(k)$ (as $\phi$ must take $F$ to an $i$-face of $H$). 
So $\gamma'(k) \notin \gamma(A)$, and in particular if $\gamma(F)$ is a face of $H$, then it must be an $i$-face.

To see that $\gamma(F)$ is an $i$-face of $H$, note that $\gamma'(A) = \gamma(A) = \phi(A)$ is a face of $H$ and by Lemma~\ref{rlk-cointerval} we have that $H' = \rlk_{H}(\phi(A)) = \rlk_{H}(\gamma(A)) $ is cointerval.
By definition, $H'$ contains both $\phi(k)$ and $\gamma(k)$.
Additionally $\rlk_{H'}(\gamma(k))$ contains $\gamma(B)$. 
Then, as $\phi(k) < \gamma(k)$, we must have $\gamma(B) \in \rlk_{H'}(\phi(k))$. 
It thus follows that $\gamma'(F) = \gamma(A) \cup \phi(k) \cup \gamma(B)$ is a face of $H$. 

It remains only to show that for each $i \in [m]$, the degree of $x_i$ in the monomial label $x_{\gamma'}$ of $\gamma'$ is no greater than $\alpha_i$; that is, that $|(\gamma')^{-1}(i)| \leq \alpha_i$. 
If $i \neq \phi(k)$, then $|(\gamma')^{-1}(i)|  \leq |\gamma^{-1}(i)|  \leq \alpha_i$. 
On the other hand let $j = \phi(k)$, so $|(\gamma')^{-1}(j)| = |\gamma^{-1}(j)|  + 1$. 
Suppose $\gamma(l) = j$. 
If $l > k$, $\gamma(l) \geq \gamma(k) > j$, a contradiction. 
So $l < k$,  and thus $\phi(l) = \gamma(l) = j$, and in particular it follows that $\gamma^{-1}(j) \subseteq \phi^{-1}(j) -  \{k\}$. 
Thus $|(\gamma')^{-1}(j)| \leq |\phi^{-1}(j)| \leq \alpha_j$.
\end{proof}
\end{lemma}

Our inductive proof depends on our ability to remove one vertex at a time from our complex, in revlex order, yet only affect one facet with this act.
Lemma~\ref{max} allows us to do this.

\begin{lemma}[Removal Lemma] \label{max} 
Let $\Delta$ be as in Theorem \ref{main}, and let $\phi$ be the vertex of $\Delta$ with the smallest label in the revlex order. 
If $\phi$ is not the only vertex of $\Delta$, then $\phi$ is properly contained in a unique facet of $\Delta$.

\begin{proof}  
For $i\in [n]$ and $l \in [m]$, define $\phi_{i,l}$ by
$$
\phi_{i,l}(x) = \begin{cases} \phi(x), & x \neq i\\ l, & x=i\end{cases}.
$$
Now for each $i$ let $S_i = \{ l : \phi_{i,l} \in \Delta \}$. Note that $\phi(i) \in S(i)$.

We claim that $\tau = (S_i)_{i \in [n]}$ defines a cell of $\Delta$. It is clear that if this is the case, any cell of $\Delta$ containing $\phi$ must be contained in $\tau$. It suffices to check that each vertex of $\tau$ is a vertex of $\Delta$ (the condition on the degree of the monomial label of $\tau$ is automatically satisfied, as the label of $\tau$ is the least common multiple of the labels of its vertices).

Let $\gamma$ be a vertex of $\tau$. Then for each $i$, $\gamma(i) = l$ where  $\phi_{i,l} \in \Delta$. It follows from the order preserving property of $\phi_{i,l}$  that $\phi(i-1) \leq l \leq \phi(i)$ (where the second inequality follows from the fact that the monomial label for $\phi_{i,l}$ must come before the label of $\phi$ in the revlex order). In particular it follows that  $x_{\gamma} \geq x_{\phi} \geq x^{\beta}$ in the revlex order.

We now need to show that $\gamma$ defines an ordered simplicial homomorphism from $G$ to $H$. We will induct on the number $r$ of elements of $[n]$ on which $\gamma$ and $\phi$ differ. If $r= 0$ or $1$ the result is immediate from the definition of $S(i)$. So suppose $r >1$. Let $l$ be the smallest index on which $\gamma$ and $\phi$ disagree. Define $\gamma'$ by 
$$
\gamma'(x) = \begin{cases} \gamma(x), & x \neq l\\ \phi(l), & x=l\end{cases}.
$$
Then $\gamma'$ is a vertex of $\tau$ that differs from $\phi$ in $r-1$ places, and thus is a vertex of $\Delta$. Furthermore, the first place in which $\gamma'$ and $\phi_{l, \gamma(l)}$ disagree is at $l$, and $\gamma(l) < \phi(l) =\gamma'(l)$, so by Lemma \ref{swap}, $\gamma$ is in $\Delta$, as desired.

Finally we claim that there is some $i$ such that $S(i) - \phi(i)$ is nonempty; this ensures that the containment $\phi \in \tau$ is proper. Let $\gamma$ be the revlex largest vertex of $\Delta$ (by assumption it must be distinct from $\phi$.) Let $k$ be the smallest index element of $[n]$ such that $\gamma(k) \neq \phi(k)$. If $\phi(k) < \gamma(k)$, we can take $\gamma'$ to be as in Lemma \ref{swap} to obtain a vertex of $\Delta$ which occurs before $\gamma$ in the revlex order, a contradiction. Thus we must have $\phi(k) > \gamma(k)$, so, by Lemma \ref{swap}, $\phi_{k,\gamma(k)}$ is a vertex of $\Delta$, and thus $S(k) - \phi(k)$ is nonempty.

\end{proof}

\end{lemma}

\begin{proof}[First proof of Theorem \ref{main}] 
Let $\Delta$ be as in the statement of the theorem. We will induct on the number of vertices of $\Delta$. If $\Delta$ is a single vertex, we are done. Now assume $\Delta$ has at least 2 vertices. We may assume the revlex smallest label of a vertex is $x^{\beta}$; let $x^{\omega}$ be the immediate predecessor of $x^{\beta}$ among the labels of vertices of $\Delta$. By induction $\Delta' = \ohom(G,H)_{\alpha}^{\geq \omega}$ is contractible.

By Lemma \ref{max}, there is a unique facet of $\Delta$ properly containing the vertex $\phi$ labelled by $x^{\beta}$. Then it follows from  \cite[Lemma 6.4]{NagelReiner}  that $\Delta$ is homotopy equivalent to the complex obtained from $\Delta$ by deleting $\phi$ and all cells containing it, that is, $\Delta'$.  Thus $\Delta$ is contractible.

\end{proof}

%%%%%

\subsection{Proof by closure operators}

Our second proof of Theorem~\ref{main} uses closure operators, a standard tool for homomorphism complexes that was also used by Dochtermann and Engstr\"{o}m \cite{DochtermannEngstromCellular}.
Recall that a poset map $f:P\rightarrow P$ is an \emph{ascending closure operator} if $f(x)\geq x$ for every $x\in P$ and $f^2=f$, while $f$ is a \emph{descending closure operator} if $f(x)\leq x$ for every $x\in P$ and $f^2=f$.  
It is known \cite{BjornerSurvey} that ascending and descending closure operators induce deformation retracts from the order complex of $P$ to the order complex of $f(P)$.
As the order complex of the face poset of a polyhedral cell complex is isomorphic to the barycentric subdivision of the complex, and hence is homeomorphic to the original complex, we may freely work on the level of order complexes for our topological arguments.

\begin{proof}[Second proof of Theorem~\ref{main}]

Let $\Delta$ be as in the theorem statement.  
We will show that $\Delta$ is contractible through $n$ pairs of ascending and descending closure operators that reduce the face poset of $\Delta$ to a single element.  
If $\alpha_i = 0$ for any $i$, by Lemma~\ref{induced-cointerval} we may reduce the problem to that coming from the complex $\ohom(G, H - i)$, so we may assume without loss of generality that $\alpha_i \geq 1$ for each $i \in [n]$.

Let $1_G$ and $1_H$ be the first elements of the vertex sets of $G$ and $H$, respectively.
Our first step is to map the face poset of $\Delta$ to the subposet of $\Delta$ whose elements are multihomomorphisms $\psi$ satisfying the condition that $\psi(1_G) = \{1_H\}$.
We do this through two closure operators.
Let $U_1(\psi)$ be defined by 
\[U_1(\psi)(i):=\left\{ 
\begin{array}{ll}
\psi(1_G)\cup\{1_H\} & \mathrm{if} \phantom{.} i=1_G \\
\psi(i)  & \mathrm{otherwise}
\end{array}
\right.
\, . \]
Note that $U_1$ is well-defined, since $\rlk(1_H)$ contains $\rlk(j)$ for every $j>1_H$, hence we can replace any simplex in $H$ of the form $\{v_1,v_2,\ldots,v_k\}$ with the simplex $\{1_H,v_2,\ldots,v_k\}$, and the second simplex is revlex larger than the first.
It is clear that $U_1$ is both ascending and a closure operator.
Now, let $D_1(\psi)$ be defined on the image of $U_1$ by
\[D_1(\psi)(i):=\left\{ 
\begin{array}{ll}
\{1_H\} & \mathrm{if} \phantom{.} i=1_G \\
\psi(i)  & \mathrm{otherwise}
\end{array}
\right.
\, . \]
Note that $D_1$ is well defined on the image of $U_1$ and is a descending closure operator.

We now proceed inductively to define $U_k$ and $D_k$, closure operators whose composition takes the image of $D_{k-1}$ to a subposet of $\Delta$ consisting of cells $\psi$ such that $|\psi(j)|=1$ for all $j\leq k$.  
Let $g(k)$ be the minimal element of $[m]=V(H)$ such that $g(k)\in\psi(k)$ for some cell $\psi$ in the image of $D_{k-1}$.
Let $U_k(\psi)$ be defined by 
\[
U_k(\psi)(i):=\left\{ 
\begin{array}{ll}
\psi(1)\cup\{g(k)\} & \mathrm{if} \phantom{.} i=1 \\
\psi(i)  & \mathrm{otherwise}
\end{array}
\right.
\]
and let $D_k$ be defined by
\[D_k(\psi)(i):=\left\{ 
\begin{array}{ll}
\{g(k)\} & \mathrm{if} \phantom{.} i=1,\ldots,k \\
\psi(i)  & \mathrm{otherwise}
\end{array}
\right.
\, . \]
Assuming $U_k$ and $D_k$ are well-defined, after applying these operators for each $k\in [n]$, we are left with a single cell of $\Delta$ satisfying $|\psi(j)|=1$ for all $j\in [n]$, i.e. a vertex of $\Delta$.  
Thus, we may conclude that $\Delta$ is contractible after verifying well-definedness.

Observe that if $U_k$ is well-defined, then it is clear that $D_k$ is well-defined.
Regarding well-definedness of $U_k$, elements of the domain of $U_k$ are of the form
\[\left( g(1),g(2),\ldots,g(k-1),\psi(k),\psi(k+1),\ldots,\psi(n)\right)\]
for some $\psi$.
Let 
\[\left( g(1),g(2),\ldots,g(k-1),a_k,a_{k+1},\ldots,a_n\right)\]
be a vertex in the corresponding cell, i.e. $a_i\in \psi(i)$ for all $i=k,\ldots,n$.
By the definition of $g(k)$, it follows that $g(k)\leq a_k$.
Let $S$ denote the right link of $g(k-1)$ in the right link of $g(k-2)$ in the right link of $g(k-3)$, etc, until reaching $g(1)$; since $H$ is cointerval, it follows that $S$ is cointerval.
Therefore, $\rlk(a_k)\subseteq \rlk(g(k))$ in $S$, and hence we may add $g(k)$ to $\psi(k)$ and obtain a cell in $\Delta$, all of whose vertices are revlex larger than $x^\beta$.
It follows that $U_k$ is well-defined.

\end{proof}

%%%%%%%%%%%%%%%%%%%%%%%%%%%%%%%%%%%%%

\section{Properties of cointerval complexes}\label{properties}

We now turn our attention to studying the class of cointerval simplicial complexes as geometric objects.  
The class of \textit{shifted} simplicial complexes \cite{Erdos-Ko-Rado, Kalai} is fundamental in the study of $f$-vectors of simplicial complexes.
Given a simplicial complex $H$, the process of \textit{combinatorial shifting} \cite{Erdos-Ko-Rado} produces a shifted simplicial complex with the same $f$-vector as $H$.  
It is often easier to prove combinatorial results for shifted simplicial complexes than for arbitrary complexes.

\begin{definition}
Let $H$ be a simplicial complex on vertex set $[n]$. 
We say that $H$ is \textit{shifted} if, for any face $F \in H$, any vertex $i \in F$, and any $j < i$ such that $j \notin F$, the set $(F \setminus \{i\})\cup\{j\}$ is a face of $H$. 
\end{definition}

Every shifted simplicial complex is cointerval.  
There are, however, cointerval complexes that are not shifted, e.g. the complex on $[4]$ with maximal faces $\{1,3\},\{2,4\}$, and $\{1,4\}$ is cointerval but not shifted under any labeling of the vertices.
Despite this, we aim to prove that cointerval simplicial complexes share some of the same ``nice'' properties that are satisfied by shifted simplicial complexes.  
Specifically, we will show that the family of cointerval simplicial complexes is closed under taking links and deletions, and that any cointerval simplicial complex is vertex-decomposable.

\subsection{Induced subcomplexes and links}

We now introduce several lemmas regarding induced subcomplexes and links in cointerval complexes.

\begin{lemma} \label{induced-cointerval} Let $H$ be a cointerval d-complex on $n$ vertices and $H'$ an induced subcomplex. Then $H'$ is cointerval.

\begin{proof}

We induct on $n$.  The result is obvious when $n \leq 2$, so suppose $n \geq 3$.  It is sufficient to prove that if $H$ is a cointerval simplicial complex on $[n]$ and $k \in [n]$, then $H-k$ is cointerval. 

Choose $i,j \in [n]-k$ with $i<j$.  We observe that $\rlk_{H-k}(i) = \rlk_H(i)-k$.  Since $\rlk_H(i)$ is a cointerval complex and $|V(\rlk_H(i))| < n$, we see that $\rlk_H(i)-k$ is cointerval by our inductive hypothesis.  Similarly, since $H$ is cointerval, we have $\rlk_H(j) \subseteq \rlk_H(i)$ and hence $$\rlk_{H-k}(j) = \rlk_H(j)-k \subseteq \rlk_H(i)-k = \rlk_{H-k}(i).$$  Thus $H-k$ is cointerval.

\end{proof}

\end{lemma}

\begin{lemma} \label{rlk-cointerval} 
Let $H$ be cointerval, $\tau \in H$. Then $\rlk_H(\tau)$ is cointerval.

\begin{proof} 
This follows inductively from Proposition~\ref{link-property} and the property that $\rlk_{H}(i)$ is cointerval for all $i \in H$.
\end{proof}

\end{lemma}

\begin{lemma} \label{links-cointerval}

Let $H$ be a cointerval simplicial complex of dimension $d$ on $[n]$.  Then $\lk_H(\tau)$ is cointerval for any face $\tau \in V(H)$.

\begin{proof}
By Proposition \ref{link-property}, it is sufficient to show that the link of any vertex $i \in H$ is cointerval. Let $G:= \lk_{H}(i)$, and suppose $V(G)$ is nonempty.  We prove the claim by induction on $d$ and on $i$.  When $d=0$ or $d=1$, the result is obvious for all $i \in [n]$. 

Suppose now that $d \geq 2$, and that the link of a vertex in a cointerval simplicial complex of dimension at most $d-1$ is cointerval.  When $i=1$, we know that $\lk_{H}(1) = \rlk_{H}(1)$, which is cointerval by definition.  Next suppose that $i>1$, and that $\lk_{H}(p)$ is cointerval for all $p<i$.  Suppose that $V(G) = \{i_1<i_2<\cdots<i_k\}$ and that $i_{j-1} < i < i_j$.  

In order to show that $G$ is cointerval, we need to show that $\rlk_{G}(i_{\ell})$ is cointerval for all $1 \leq \ell \leq k$ and that if $i_{\ell} < i_m$, then $\rlk_{G}(i_m) \subseteq \rlk_{G}(i_{\ell})$.  First we show that $\rlk_{G}(i_{\ell})$ is cointerval by examining two cases.  If $\ell \geq j$, then $\rlk_{G}(i_{\ell}) = \rlk_{\rlk_{H}(i)}(i_{\ell})$.  Since $\rlk_{H}(i)$ is cointerval, $\rlk_{G}(i_{\ell})$ is cointerval as well.  On the other hand, if $\ell < j$, then $\rlk_{G}(i_{\ell}) = \lk_{\lk_{H}(i)}(i_{\ell})|_{\{i_r>i_{\ell}\}} = \lk_{\lk_{H}(i_{\ell})}(i)|_{\{i_r > i_{\ell}\}}$ by Proposition \ref{link-property}.   By our inductive hypothesis on $i$, $\lk_{H}(i_{\ell})$ is cointerval of dimension at most $d-1$.  By induction on $d$, $\lk_{\lk_{H}(i_{\ell})}(i)$ is also cointerval.  Thus by Lemma \ref{induced-cointerval}, $\rlk_{G}(i_{\ell})$ is cointerval.

Finally suppose $i_{\ell} < i_m$, and let $F$ be a face in $\rlk_{G}(i_m)$.  We show that $F \in \rlk_{G}(i_{\ell})$ by examining three cases.  

If $i < i_{\ell} < i_m$, then $F \in \rlk_{\rlk_H(i)}(i_m) \subseteq \rlk_{\rlk_H(i)}(i_{\ell})$ since $\rlk_H(i)$ is cointerval.  Thus $F \in \rlk_{G}(i_{\ell})$.  

If $i_{\ell} < i < i_m$, then ${i_m} \cup F \in \rlk_H(i)$.  Since $i_{\ell}<i$ and $H$ is cointerval, ${i_m} \cup F \in \rlk_{H}(i_{\ell})$ as well.  Finally, since $i,i_m$ are vertices of the cointerval complex $\rlk_H(i_{\ell})$, and $F \in \rlk_{\rlk_H(i_{\ell})}(i_m) \subseteq \rlk_{\rlk_H(i_{\ell})}(i)$, we see that $F \in \rlk_{G}(i_{\ell})$. 

If $i_{\ell}<i_m<i$, then $F \cup i \in \rlk_H(i_m) \subseteq \rlk_H(i_{\ell})$ (again since $H$ is cointerval).  Thus $F \in \rlk_G(i_{\ell})$. 
\end{proof}
\end{lemma}

%%%

\subsection{Vertex-Decomposability}

Our next goal is to study the class of vertex-decomposable simplicial complexes.  Roughly speaking, vertex-decomposable complexes are inductively constructed from a simplex by attaching well-behaved simplicial cones. 

\begin{definition}
A simplicial complex $H$ is \textit{vertex-decomposable} if 
\begin{itemize}
\item $H$ is a simplex or $H = \{\emptyset\}$, or 
\item there exists a vertex $x \in H$ such that 
\begin{itemize}
\item $H \setminus \{x\}$ and $\lk_{H}(x)$ are vertex decomposable, and 
\item no facet of $\lk_{H}(x)$ is a facet of $H \setminus \{x\}$. 
\end{itemize}
\end{itemize}
\end{definition}

If $H$ is vertex-decomposable, the specified vertex $x$ from the definition is called a \textit{shedding vertex} for $H$.  The concept of vertex-decomposability for non-pure simplicial complexes was introduced by Bj\"orner and Wachs \cite{BjornerWachsI, BjornerWachsII}.  They prove that any shifted simplicial complex is vertex-decomposable, and that any vertex-decomposable complex is (non-pure) shellable (hence sequentially Cohen-Macaulay).  Since shifted simplicial complexes provide a large class  of cointerval simplicial complexes, it is natural to ask whether or not cointerval complexes are also vertex-decomposable.  

\begin{theorem}\label{decomposable}
Let $H$ be a $(d-1)$-dimensional cointerval complex on $[n]$.  Then $H$ is vertex-decomposable. 
\end{theorem}

\begin{proof}
We prove the claim by induction on $d$ and $n$.  
The empty complex is vertex-decomposable by definition, and it is clear that any $0$-dimensional cointerval complex is vertex decomposable, so suppose $d >1$.  
The only $(d-1)$-dimensional complex on $n=d$ vertices is a simplex, which is vertex decomposable by definition.  

Suppose now that $n > d$, and that any cointerval complex $\Gamma$ with $\dim H' < \dim H$ or $\dim H' = \dim H$ and $|V(H')| < |V(H)|$ is vertex-decomposable.  
We will use the reverse-lexicographic (revlex) total order on $2^{[n]}$, in which $F<_{revlex}G$ if $\max\{F\setminus G \cup G \setminus F\} \in G$. 
Let $\sigma \subset [n]$ be the revlex smallest subset of $[n]$ that is not a face in $H$ (i.e. $\tau \in H$ for all $\tau <_{revlex} \sigma$, but $\sigma \notin H$).  
Since $n>d$, $H$ is not a simplex, and such a subset $\sigma$ exists.  
Let $j$ denote the largest element of $\sigma$, and notice that $j>1$ since $H$ is nonempty. 
We claim that $j$ is a shedding vertex for $H$.  

By Lemma \ref{links-cointerval}, $\lk_{H}(j)$ is cointerval, and $H \setminus\{j\}$ is cointerval since it is the restriction of $H$ to $[n]\setminus \{j\}$.  
Thus we need only show that no facet of $\lk_{H}(j)$ is a facet of $H \setminus \{j\}$.  
Suppose that $F$ is a facet of $\lk_{H}(j)$ and partition $F = F_L \cup F_R$, where $F_L = F \cap \{1, \ldots, j-1\}$, and $F_R = F \cap \{j+1, \ldots, n\}$.  

Observe that $F_L$ must be a proper subset of $\{1,\ldots,j-1\}$.  Indeed, if $F_L = \{1,\ldots,j-1\}$, then we would have $\{1,\ldots,j\} \subseteq \{j\} \cup F \in H$, and hence $\{1,\ldots,j\} \in  H$.  However $\sigma \subseteq \{1,\ldots,j\}$, which would contradict our assumption that $\sigma$ is not a face in $H$.  Thus there is some vertex $i \in \{1,\ldots,j-1\} \setminus F_L$.  

Now we consider $\lk_{H}(F_L)$, which is cointerval by Lemma \ref{links-cointerval}.  Since $\{1,\ldots,j-1\} <_{revlex} \sigma$, we see that $\{1,\ldots,j-1\} \in H$, and hence $F_L \cup \{i\} \in H$ as well.  Thus $i$ is a vertex in $\lk_{H}(F_L)$.  Moreover, $\{j\} \cup F_R \in \lk_{H}(F_L)$ as well, and $i < j$.  Since $\lk_{H}(F_L)$ is cointerval, it follows that $\{i\} \cup F_R \in \lk_{H}(F_L)$ as well.  Thus $\{i\} \cup F \in H \setminus\{j\}$, as desired. 

\end{proof}

\begin{remark}
The family of matroid complexes is also closed under links and deletions, so it is again natural to ask if matroid complexes are cointerval.  
This question has a negative answer in general, as it is straightforward to check case-by-case and see that the Fano and anti-Fano matroids cannot have a vertex ordering producing a cointerval complex.
However, it is also straightforward to show that matroid complexes of rank two, i.e. complete multipartite graphs, are cointerval under a suitable ordering.
\end{remark}

%%%%%%%%%%%%%%%%%%%%%%%%%%%%%%%%%%%%%%%%%%%%%%%%%%%%%%%%%%%

\section{Nonnesting monomial ideals}\label{families}

In this section, we introduce \emph{nonnesting monomial ideals} and apply the homomorphism approach to this family.
Our purpose is to illustrate how an ideal can have an underlying structure, one defined by simplicial homomorphisms, that is not immediately clear from its definition.
We begin by recalling the definition of nonnesting partitions; these objects were defined by Postnikov and studied by Athanasiadis \cite{AthanasiadisNonnesting}.
We refer the reader to Armstrong's memoir \cite[Chapter 5]{ArmstrongMemoir} for a nice exposition regarding nonnesting partitions.

\begin{definition}
For $r\in \mathbb{Z}_{\geq 1}$, let $\nP:=\{P_1,P_2,\ldots,P_m\}$ denote a partition of the set $[r]$.
We say that two blocks $P_i\neq P_j$ \emph{nest} if there exist $1\leq a<b<c<d\leq r$ with $\{a,d\}\subseteq P_i$ and $\{b,c\}\subseteq P_j$ and there does not exist $e\in P_i$ with $b<e<c$.
If no pair of the blocks of $\nP$ nest, we say that $\nP$ is a \emph{nonnesting partition} of $[r]$.
\end{definition}

\begin{example}
The partition $\{\{1,4\},\{2,5,6\},\{3\}\}$ of $[6]$ is a nonnesting partition, while the partition $\{\{1,3,5\},\{2,6\},\{4\}\}$ is not since the first two blocks nest.
\end{example}

The number of nonnesting partitions of $[r]$ is given by the Catalan number $\frac{1}{r}\binom{2r}{r-1}$.
It is convenient to represent a nonnesting partition as a graph as follows.

\begin{definition}
Let $\nP$ be a partition of $[r]$ and let $P_i=\{i_1,i_2,\ldots,i_k\}$ denote a block of $\nP$.
The \emph{arc diagram} $G_\nP$ of $\nP$ is the graph on vertex set $[r]$ whose edges are given by $\{i_j,i_{j+1}\}$ for consecutive elements of $P_i$, taken over all blocks of $\nP$.
\end{definition}

The name ``arc diagram'' comes from the picture of this graph obtained when the elements of $[r]$ are placed in a line and the edges of the diagram are drawn as upper semicircular arcs, as in Figure~\ref{bumpfig}.
The idea of using an arc diagram goes back to Postnikov's connection between nonnesting partitions and antichains in type $A$ root posets, though there is not a standard term for such diagrams in the literature.
In this representation, a partition is nonnesting exactly when no arc is nested below another.

\begin{figure}[ht]
\begin{center}
\includegraphics{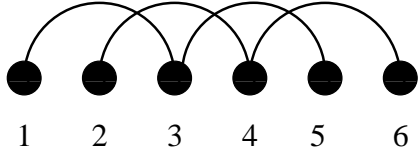}
\end{center}
\caption{The arc diagram for $\{\{1,3,5\},\{2,4,6\}\}$.}
\label{bumpfig}
\end{figure}

To connect nonnesting partitions to monomial ideals, we observe that a partition of $[r]$ may be used to define a family of monomial ideals as follows.

\begin{definition}
Let $\nP$ be a nonnesting partition of $[r]$ and $n\in \mathbb{Z}_{\geq 1}$.
The \emph{nonnesting monomial ideal restricted by $\nP$}, denoted $I_\nP(n) \subseteq \field[x_1,\ldots,x_n]$, is the unique (if it exists) monomial ideal whose minimal generators are degree $r$ monomials $x_{i_1}x_{i_2} \cdots x_{i_r}$, where $i_1\leq i_2 \leq\cdots\leq i_r$, such that if $a$ and $b$ are in the same block of $\nP$, then $x_{i_a}\neq x_{i_b}$.
\end{definition}

A nonnesting monomial ideal is thus a monomial ideal whose generators are all monomials of fixed degree that restrict to squarefree monomials on the blocks of $\nP$, when the variables for the generators are listed by increasing subscripts.
For the partition of $[r]$ into single-element sets, $I_\nP(n)$ is $\mathfrak{m}^r$, the $r$-th power of the maximal ideal $\mathfrak{m}=(x_1,x_2,\ldots,x_n)\subseteq \field[x_1,\ldots,x_n]$.
For the partition of $[r]$ as a single set $[r]$, $I_\nP(n)$ is generated by the squarefree generators of $\mathfrak{m}^r$.
One may view the family of nonnesting monomial ideals as interpolating between these two well-studied nonnesting monomial ideals.

Our goal in this section is to illustrate how minimal cellular resolutions of nonnesting monomial ideals arise as homomorphism complexes.
The following proposition reveals that $I_\nP(n)$ is defined by homomorphisms from $G_\nP$ to $K_n$.

\begin{proposition}
The nonnesting monomial ideal $I_\nP(n)$ is equal to the ordered homomorphism ideal $I_{G_\nP,K_n,ord}$.
\end{proposition}

\begin{proof}
Let $x=x_{i_1}x_{i_2} \cdots x_{i_r}$, where $i_1\leq i_2 \leq\cdots\leq i_r$, be a generator of $I_\nP(n)$.
Consider the function $\phi_x:[r]\rightarrow [n]$ given by $\phi_x(j)=i_j$.
That $\phi_x$ preserves edges follows immediately from the definition, since if $x_{i_a}\neq x_{i_b}$, then $i_a\neq i_b$.
That $\phi_x$ is order-perserving is also immediate from the stated form of $x$.
Thus, $\phi_x$ is a nondegenerate ordered simplicial homomorphism from $G_\nP$ to $K_n$.

Let $\phi$ be such a simplicial homomorphism, and we will show that the monomial $x_\phi$ satisfies the restrictions given by $\nP$.
Each connected component of $G_\nP$ is a path whose vertices are the elements of a block of $\nP$.
As $\phi$ is order-preserving and edge-preserving, we know that on each path in $G_\nP$, the set of vertices of the path are sent to an increasing set of distinct elements of $[n]$, hence to a set of distinct indices of variables in $x_\phi$.
This verifies that on each block of $\nP$, the variables in the positions specified by that block are distinct.
\end{proof}

Arc diagrams of nonnesting partitions are the only graphs needed when studying $I_{G,K_n,ord}$ for an arbitrary graph $G$, as the following proposition shows.

\begin{proposition}\label{reducedgraph}
For any graph $G$ on vertex set $[r]$, there exists a unique nonnesting partition $\nP$ of $[r]$ such that $G_\nP$ is a subgraph of $G$ and $I_{G,K_n,ord}=I_{G_\nP,K_n,ord}$.
We call $G_\nP$ the \emph{reduced graph} for $G$.
\end{proposition}

\begin{proof}
We induct on $|E(G)|$.  The result is clear when $|E(G)| \leq 1$, so suppose that $|E(G)| \geq 2$.  
If $G$ is already nonnesting, we set $G = G_{\mathcal{P}}$; otherwise, suppose there are vertices $a<b<c<d$ for which $\{a,d\}, \{b,c\} \in E(G)$.

Consider the graph $G'$ obtained from $G$ by removing the edge $\{a,d\}$.  
We claim that $I_{G,K_n,ord} = I_{G',K_n,ord}$.  
Indeed, suppose $\phi: G' \rightarrow K_n$ is an ordered simplicial homomorphism.  
It follows that $\phi(a) \leq \phi(b) < \phi(c) \leq \phi(d)$, and hence $\phi(a) < \phi(d)$.  
Thus $\phi$ extends to an ordered simplicial homomorphism from $G$ to $K_n$.  
Similarly, any ordered simplicial homomorphism $\psi: G \rightarrow K_n$ restricts to an ordered simplicial homomorphism from $G'$ to $K_n$, proving the claim.

By our inductive hypothesis, there is a nonnesting partition $\mathcal{P}$ for which $I_{G',K_n,ord} = I_{G_{\mathcal{P}},K_n,ord}$.  
Hence $I_{G,K_n,ord} = I_{G_{\mathcal{P}},K_n,ord}$ as well.
\end{proof}

It follows immediately from the techniques developed in Section~\ref{resolutions} that nonnesting monomial ideals admit minimal prodsimplicial resolutions.
It is natural to investigate algebraic invariants of $I_{G_\nP,K_n,ord}$; the rest of this section regards the Betti numbers of $I_{G_\nP,K_n,ord}$.
We begin by establishing that the nonnesting partitions on $[r]$ admit a natural partial order.
We then proceed to describe how the Betti numbers of $I_{G_\nP,K_n,ord}$ might be computed using this poset.
The poset $\nD_r$ that we define below arises in various guises for different Catalan-enumerated families \cite{BarcucciEtAl,CautisJackson,FerrariPinzani,SapounakisEtAl}.
In particular, $\nD_r$ is a distributive lattice that is isomorphic to the lattice of $312$-avoiding permutations in $\mathfrak{S}_r$ under weak Bruhat order, or equivalently the lattice of Dyck paths, i.e. lattice paths from $(0,0)$ to $(r,r)$ using North and East steps that do not drop below the line $x=y$, under the order given by geometric inclusion of paths.

\begin{definition}
The \emph{$r$-th diagram poset}, denoted $\nD_r$, is the poset whose elements are arc diagrams of nonnesting partitions of $[r]$ partially ordered by $\nP\leq\nQ$ if every arc in $\nQ$ lies above an arc in $\nP$.
\end{definition}

It is straightforward to check that the minimal element of $\nD_r$ is the full path on $[r]$ and the maximal element is the empty graph.
It is clear from the definition of minimal cellular resolution that the number of $k$-dimensional faces in $\ohom(G,H)$ is the $k$-th Betti number of $I_{G,H,ord}$ when $H$ is cointerval.
With this in mind, the importance of this poset is illustrated by the following.

\begin{proposition}
If $\nP\leq \nQ$ in $\nD_r$, then 
\[\ohom(G_\nP,K_n)\subseteq\ohom(G_\nQ,K_n) \, .\]
Further, for each face $\tau$ in $\ohom(G_{e},K_n)$, where $G_e$ denotes the empty graph, the upper order ideal $U(\tau)$ in $\nD_r$ of arc diagrams whose $\ohom$ complexes contain $\tau$ is principal.
\end{proposition}

\begin{proof}
Let $\tau = (\tau_1,\ldots,\tau_r)$ be a cell in $\ohom(G_{\nP},K_n)$ so that each choice $\varphi(i) \in \tau_i$ yields an ordered homomorphism from $G_{\nP}$ to $K_n$. 
We need to show that each such choice yields an ordered homomorphism from $G_{\nQ}$ to $K_n$ as well.

For each arc $\{a<d\} \in G_\nP$, there is an arc $\{b<c\} \in G_\nQ$ with $a \leq b < c \leq d$.  
Since $\varphi$ is an ordered homomorphism from $G_{\nP}$ to $K_n$, we see that $\varphi(a) \leq \varphi(b) < \varphi(c) \leq \varphi(d)$.  
Thus $\varphi(a) < \varphi(d)$ for all arcs $\{a<d\} \in G_{\nQ}$, and hence $\varphi$ is an ordered homomorphism from $G_{\nQ}$ to $K_n$.  
Thus $\tau \in \ohom(G_{\nQ},K_n)$, as desired.

Suppose next that $\tau$ is a face of $\ohom(G_e,K_n)$.  
Consider the graph $G$ obtained as the union of all reduced graphs $G_{\nQ}$ such that $\tau$ is a face of $\ohom(G_{\nQ},K_n)$.  
By Proposition \ref{reducedgraph}, there is a unique nonnesting partition $\nP$ such that $G_\nP$ is a subgraph of $G$ and $\ohom(G,K_n) = \ohom(G_{\nP},K_n)$.  
Since $G_{\nP}$ is obtained from $G$ by removing nesting arcs, it follows that $\nP \leq \nQ$ for all nonnesting partitions $\nQ$ whose $\ohom$ complexes contain $\tau$.  
Thus the upper order ideal $U(\tau)$ in $\nD_r$ is generated by $\nQ$.
\end{proof}

We next show that one may compute the Betti numbers of $I_\nP(n)$ via weighted sums ranging over subsets of the elements of $\nD_r$.

\begin{definition}
For $\nP$ a nonnesting partition, define the \emph{$k$-th weight} of $\nP$ for $n$, denoted $\omega_k(\nP,n)$, to be the number of $k$-dimensional faces $\tau$ of $\ohom(G_e,K_n)$ such that $\nP$ is the generator of $U(\tau)$.
\end{definition}

Weights are thus the number of faces that are determined ``minimally'' by $\nP$ in the $\ohom$ complex.
The following connection between the $k$-th weights of $\nP$ for $n$ and the Betti numbers of $I_\nP(n)$ follows immediately from the definitions.

\begin{proposition}
\[\beta_k(I_\nQ(n))=\sum_{\nP\leq \nQ}\omega_k(\nP,n) \, .\]
\end{proposition}

From the M\"{o}bius inversion formula \cite[Chapter 3]{StanleyVol1} it follows that the $k$-th weights and the $k$-th Betti numbers of all nonnesting partitions of $[r]$ determine each other uniquely.
As the following theorem demonstrates, the $k$-th weights of $\nP$ are non-zero for a restricted class of nonnesting partitions.

\begin{theorem}\label{zeroweight}
If $G_\nP$ contains an arc from $i$ to $j$ where $j-i>2$, then $\omega_k(\nP,n)=0$.
\end{theorem}

\begin{proof}
Suppose for contradiction that $\tau$ is a face of $\ohom(G_e,K_n)$ with $U(\tau)$ generated by $\nP$.
Suppose that $G_\nP$ contains an arc $\{i,j\}$ with $j-i>2$.
By definition, $\tau(i)\cap\tau(j)=\emptyset$.
If $y$ is the maximal element of $\tau(i)$, then $\tau(i)\cap\tau(i+s)=\{y\}$ for $s=1,\ldots,j-1$, since every homomorphism corresponding to a vertex in $\tau$ must preserve order.
It follows that $\tau(j-2)=\{y\}$.
Because $y\notin\tau(j)$, adding the edge $\{j-2,j\}$ to $G_\nP$ creates a new graph $G$ that reduces to an arc diagram $G_\nQ\in U(\tau)$ with $\nQ\leq \nP$.
This contradicts the minimality of $\nP$ as a generator of $U(\tau)$.
\end{proof}

We will call an arc diagram $G_\nP$ containing no arc of the form $\{i,j\}$ with $j-i>2$ a \emph{small arc diagram}.
The number of small arc diagrams on $[r]$ is an even-index Fibonacci number, as the following theorem shows.

\begin{theorem}
Let $F_{r}$ denote the Fibonacci numbers with $F_0=F_1=1$.
The number of small arc diagrams on $[r]$ is $F_{2r-2}$.
\end{theorem}

\begin{proof}
Let $f(r)$ denote the number of small arc diagrams on $[r]$.
We proceed by induction on $r$, where the cases $r=1$ and $r=2$ are clear.
Assume that $f(r)=F_{2r-2}$ for all $r\leq n-1$.
To construct a small arc diagram $G$ on $[n]$, there are two possibilities: either $\{1,2\}$ is an edge or it isn't. 
If it is, there are $f(n-1)$ ways to fill in the graph on the vertices $[2,n]$.
If not, we add edges of the form $\{i,i+2\}$ until we reach an integer $K$ such that all edges $\{1,3\}, \{2,4\},\ldots,\{K-2,K\}$ are in $G$, but $\{K-1,K+1\}$ is not, where here $K\geq 2$.
Given such a graph, there are $f(n-K+1)$ ways to fill in the edges on $[K,n]$.
We thus obtain a recursion 
\[f(n) = f(n-1) + (f(n-1)+\cdots +f(1))=F_{2n-4}+(F_{2n-4}+\cdots+F_{0})=F_{2n-4}+F_{2n-3} \, ,\] 
which completes our proof.
\end{proof}

\begin{remark}
There are other Catalan-enumerated families with substructures enumerated by the even-indexed Fibonacci numbers, for example \cite[Corollary 4.4]{TennerPattern} the permutations in the symmetric group avoiding the patterns $321$ and $3412$.
\end{remark}

While in general it appears to be nontrivial to compute $\omega_k(\nP,n)$, the weights have clean formulas when $k=0$.

\begin{proposition}
Let $\nP$ be a nonnesting partition on $[r]$.
If $G_\nP$ contains an arc that is not of the form $\{i,i+1\}$ for some $i$, then $\omega_0(G_\nP,n)=0$.
If $G_\nP$ only has arcs given by consecutive elements of $[r]$, then
\[\omega_0(G_\nP,n)=\binom{n}{\hat{c}_\nP} \]
where $\hat{c}_\nP$ is the number of connected components of the complement of $G_\nP$ in the full path graph on $[r]$ with edges given by $\{i,i+1\}$ for all $i$.
\end{proposition}

\begin{proof}
Suppose that $G_\nP$ has an arc of the form $\{i,j\}$ for $i+1<j$.  
Let $g$ satisfy $i<g<j$.
Then for any ordered homomorphism $\phi$ from $G_\nP$ to $K_n$, $\phi(g)$ is either equal to $i$, equal to $j$, or equal to neither.
Thus, the reduced graph for $\phi$ given by Proposition~\ref{reducedgraph} is less than $G_\nP$ in the Dyck path poset, and hence $\phi$ has already been accounted for.

Now, suppose that $G_\nP$ only has arcs given by consecutive elements of $[r]$.
If $\phi$ is an ordered homomorphism from $G_\nP$ to $K_n$, then for $\phi$ to not have been previously counted in the weight for a graph lower than $G_\nP$ in the Dyck path poset, $\phi$ must have $\phi(i)=\phi(i+1)$ if and only if $\{i,i+1\}\notin E(G_\nP)$.
Every homomorphism of this kind arises as an ordered labeling of $[r]$ that is constant on connected components of the complement of $G_\nP$ in the path graph, and the proposition follows.
\end{proof}

%\begin{corollary}
%The $0$-th Betti number of $I_\nP(n)$ is given by a positive integral combination of the binomial coefficients $\binom{n}{k}$ for $k$ less than or equal to $\hat{c}_{\nP}$.
%\end{corollary}

\begin{remark}
In general, the $k$-th Betti numbers and the $k$-th weights for arbitrary $\nP$ appear difficult to express with a simple formula.
It would be interesting to determine whether or not these values have broader combinatorial or algebraic significance.
\end{remark}

%%%%%%%%%%%%%%%%%%%%%%%%%%%%%%%%%%%%%%%%

\bibliography{Braun}
\bibliographystyle{plain}

\end{document}